\documentclass[orivec]{llncs}
\usepackage{amssymb}
\usepackage{amsmath}

  \newcommand{\J}{{\sf J}}
   \newcommand{\RPL}{{\sf RPL}}
   \newcommand{\F}{{\sf RPLJ}}
  \newcommand{\M}{{\mathcal M}}
  \newcommand{\E}{{\mathcal E}}
\newcommand{\CS}{{\sf CS}}
 \newcommand{\W}{{\mathcal W}}
  \newcommand{\V}{{\mathcal V}}
\newcommand{\R}{{\mathcal R}}

 \def\r{\rightarrow}

  \def\vd{\vdash}

\begin{document}

\title{Justification Logics in a Fuzzy Setting}
\author{Meghdad Ghari}
\institute{School of Mathematics,
Institute for Research in Fundamental Sciences (IPM), \\ P.O.Box: 19395-5746, Tehran, Iran \\ \email{ghari@ipm.ir}
}
\maketitle

\begin{abstract}
Justification Logics provide a framework for reasoning about justifications and evidences. Most of the accounts of
justification logics are crisp in the sense that agent's justifications for a statement is convincing or is not. In this paper,
we study fuzzy variants of justification logics, in which an agent can have a justification for a statement with a certainty degree between 0 and 1. We replaced the classical base of the justification logics with some known fuzzy logics: Hajek's basic logic, {\L}ukasiewicz logic, G\"{o}del logic, product logic, and rational Pavelka logic. In all of the resulting systems we introduced fuzzy models (fuzzy possible world semantics with crisp accessibility relation) for our systems, and established the soundness theorems. In the extension of rational Pavelka logic we also proved a graded-style completeness theorem.
\end{abstract}
{\bf Keywords:} Justification logic, Fuzzy logic, Rational Pavelka logic.\\

\section{Introduction}
Justification logics is a family of logics used to reason about justifications and reasons for our belief or knowledge. Justification logics began with Artemov's logic of proofs \cite{A1995,A2001}, however various systems for justification logics have been introduced in the literature (see e.g. \cite{A2008,KuznetsStuder2012}). Justification logics are extensions of classical propositional logic by expressions of the form $t:A$, where $A$ is a formula and $t$ is a justification term. Regarding epistemic models of justification logics (cf. \cite{Fitting2005,Mkrtychev1997}), the formula $t:A$ is interpreted  as ``$t$ is a justification (or evidence) for $A$", and some of the justification logics also enjoy arithmetical interpretation of the operator ``$:$", in which $t:A$ can be read as ``$t$ is a proof of $A$ in Peano arithmetic." Note that mathematical proofs are certain justifications for theorems. But in real life reasoning we often deal with uncertain justifications.

It seems the only work that addresses uncertain justifications for belief is Milnikel's logic of uncertain justifications  \cite{Milnikel2014}. He replaced the justification assertions $t:A$ with $t:_r A$, where $r\in\mathbb{Q}\cap(0,1]$, with the intended meaning ``I have (at least) degreee $r$ of confidence in the reliability of $t$ as evidence for belief in $A$." Nevertheless, Milnikel's logic of uncertain justifications has (two-valued) classical propositional logic as its logical part. We think that the true logic of uncertain justifications should be based on fuzzy logics. fuzzy logics are well-known frameworks with the aim of inference under vagueness.\footnote{In this paper, we use the term fuzzy logic in the \textit{narrow sense} distinguished by Zadeh in \cite{Zadeh1994} (from the \textit{broad sense}), namely fuzzy logic is a logical system which aims at a formalization of approximate reasoning.} 

 This paper tries to give a synthesis of justification logic and fuzzy logic (both systems are considered in the propositional level in this paper). We extended justification logic and fuzzy logic to \textit{fuzzy justification logic}, a framework for reasoning on justifications under vagueness.   We replaced the classical base of the justification logic with some known fuzzy logics: Hajek's basic logic, {\L}ukasiewicz logic, G\"{o}del logic, product logic, and rational Pavelka logic (cf. \cite{Hajek1998}). In all of the resulting systems we introduced fuzzy models (fuzzy possible world semantics with crisp accessibility relation) for our systems, and established the soundness theorems. In the extension of rational Pavelka logic we also proved a graded-style completeness theorem.
  
 Our systems are \textit{t-norm based}, that means given an arbitrary t-norm all connectives have associated truth degreee functions which are defined from that t-norm (cf. \cite{Gottwald2000}). The language of fuzzy justification logics includes expressions of the form $t:A$, but with a different interpretation. The formula $t:A$ can be interpreted as ``$t$ is an evidence for believing $A$ with a certainty degree taken form the interval $[0,1]$." In other words, $t$ may be an evidence for $A$ that is not fully convincing and is only plausible to some extent. This allows us to reason about those facts that we do not have complete information or knowledge about them. Moreover, in fuzzy justification logics which are extensions of rational Pavelka logic (see Section \ref{sec:Fuzzy justification logics with truth constants}) we are able to define justification assertions of the form $t:_r A$, $t:^r A$, and $t \overset{r}{:} A$ that can be read respectively as
 ``$t$ is a justification for believing $A$ with certainty degree at least $r$",
 ``$t$ is a justification for believing $A$ with  certainty degree at most $r$", and
 ``$t$ is a justification for believing $A$ with certainty degree $r$."

The paper is organized as follows. In Section \ref{sec:Basic justification logic}, we review the axioms and rules of the basic justification logic, and its semantics. In Section \ref{sec:T-norm based fuzzy logics}, we review various t-norm based fuzzy logics. In Section \ref{sec:T-norm based justification logics}, we introduce basic fuzzy justification logic and its fuzzy models. We also introduce {\L}ukasiewicz, G\"{o}del, and product justification logics and their models.  In Section \ref{sec:Fuzzy justification logics with truth constants} we will develop a fuzzy justification logic whose language contains truth constants, and prove soundness and graded-style completeness theorems.
\section{Basic justification logic}\label{sec:Basic justification logic}
In this section, we recall the axiomatic formulation of basic
justification logic {\sf J} and its semantics. The language of $\J$ is an extension of the language of propositional logic by expressions of the form $t:A$, where $A$ is a formula and $t$ is a
justification term. \textit{Justification terms} are built up from
justification variables $x_1, x_2, x_3, \ldots$ and justification
constants $c,c_1,c_2, \ldots$ using binary operations `$\cdot$' and `+'. More formally, justification terms and formulas of $\J$ are formed
by the following grammar:
\[ t::= x_i~|~c_i~|~t\cdot t~|~t+t,\]
\[ A::= p~|~\bot~|~A\rightarrow A~|~t:A,\]
where $p$ is a propositional variable and $\bot$ is the truth constant  falsity.  The set of all propositional variables, all terms and all formulas of $\J$ are denoted respectively by $\mathcal{P}$, $Tm$, and $Fm_\J$.

Next we describe the axioms and rules of \J.
\begin{definition}\label{def: axioms of J}
 Axiom schemes of {\sf J} are:
\begin{description}
\item[PC.] Finite set of axiom schemes for classical propositional calculus,
\item[Appl.]  $s:(A\r B)\r(t:A\r (s\cdot
t):B)$, 
\item[Sum.]  $s:A\r (s+t):A~,~s:A\r (t+s):A$,
\end{description}
{\bf Rules of inference}:
\begin{description}
\item[MP.] \textit{Modus Ponens}, from $\vd A$ and $\vd A\r B$,
infer $\vd B$.
 \item[IAN.] \textit{Iterated Axiom Necessitation}, $\vd c_{i_n}:c_{i_{n-1}}:\ldots:c_{i_1}:A$, where $A$ is an axiom instance of \J, $c_{i_j}$'s are
justification constants and $n\geq 1$.
\end{description}
\end{definition}

We now proceed to the definition of
\textit{Constant Specifications}.

\begin{definition}
A \textit{constant specification} $\CS$ for $\J$ is a set of formulas of
the form $c_{i_n}:c_{i_{n-1}}:\ldots:c_{i_1}:A$ ($n\geq 1$), where $c_{i_j}$'s are
justification constants and $A$ is an axiom instance of $\J$, such that it is downward closed: if $c_{i_n}:c_{i_{n-1}}:\ldots:c_{i_1}:A\in\CS$, then $c_{i_{n-1}}:\ldots:c_{i_1}:A\in\CS$.
\end{definition}
 Let ${\sf J}_\CS$ be the fragment of ${\sf J}$ where
the Iterated Axiom Necessitation rule only produces formulas
from the given $\CS$. 

In the rest of this section we introduce Fitting models for {\sf J}.
\begin{definition}\label{def:F-model J}
A Fitting model $\M=(\W,\R,\E,\V)$ for justification
logic ${\sf J}_\CS$ (or an $\J_\CS$-model for short) consists of a non-empty set of possible worlds $\W$, an accessibility relation $\R$ which is a binary relation on $\W$, i.e. $\R\subseteq \W\times\W$, a truth valuation $\V:\W\times\mathcal{P} \rightarrow \{0,1\}$ and an admissible evidence function $\E:\W\times Tm\times Fm_\J\r \{0,1\}$ meeting the following conditions (we use the notation $\E_w(t,A)$ instead of $\E(w,t,A)$):
\begin{description}
 \item[$\E 1.$] If $\E_w(s,A\r B)=1$ and $\E_w(t,A)=1$, then $\E_w(s\cdot t,B)=1$.
 \item[$\E 2.$] If $\E_w(s,A)=1$, then $\E_w(s+t,A)=\E_w(t+s,A)=1$.
  \item[$\E 3.$]  If $c:F\in\CS$, then $\E_w(c,F)=1$.
\end{description}
The truth valuation $\V$ extends uniquely to all formulas, i.e. $\V:\W\times Fm_\J \rightarrow \{0,1\}$, as follows (we use the notation $\V_w(A)$ instead of $\V(w,A)$):
\begin{enumerate}
\item $\V_w(\bot)=0$,
\item $\V_w(A\r B)=1$ if{f} $\V_w(A)=0$ or $\V_w(B)=1$,
 \item $\V_w(t:A)=1$ if{f} $\E_w(t,A)=1$ and for every $v\in \W$ with $w \R v$, $\V_v(A)=1$.
\end{enumerate}
We say that a formula $A$  is $\J_\CS$-valid (denoted by
$\J_\CS\Vdash A$) if for every $\J_\CS$-model $\M=(\W,\R,\E,\V)$ and every $w\in\W$ we have  $\V_w(A)=1$.
\end{definition}

The intended meaning of the admissible evidence function $\E$ is as follows: $\E_w(t,A)=1$ means that $t$ is an admissible justification for $A$ in $w$. Thus condition $\V3$ means: the formula $t:A$ is true in a world $w$ of a model $\M$ if and only if $t$ is an admissible justification for $A$ in $w$, and $A$ is true in all $w$-accessible worlds. 

The proof of the following theorem can be found in \cite{A2008}.
\begin{theorem}\label{M-completeness}
For a given constant specification $\CS$, the basic justification logic
${\sf J}_\CS$ is sound and complete with respect to their $\J_\CS$-models, i.e., ${\sf J}_\CS\vd A$ if{f} $\J_\CS\Vdash A$.
 \end{theorem}
 
Among various papers on the justification logics, it seems the only work that addresses uncertain justifications for belief is Milnikel's logic of uncertain justifications $\J^U$ (cf. \cite{Milnikel2014}). He replaced the justification assertions $t:A$ with $t:_r A$, where $r\in\mathbb{Q}\cap(0,1]$, with the intended meaning ``I have (at least) degreee $r$ of confidence in the reliability of $t$ as evidence for belief in $A$." Milnikel considered the following principles:
\begin{enumerate}
\item $s:_r (A\r B)\r(t:_{r'} A \r s\cdot t:_{r\cdot r'} B)$.
\item $s:_r A\r s+t:_r A$, $s:_r A\r t+s:_r A$.
\item $t:_{r'} A \r t:_r A$, where $r \leq r'$.
\end{enumerate}
along with the rule IAN with the conclusion: $\vd c_{i_n}:_1 c_{i_{n-1}}:_1 \ldots:_1 c_{i_1}:_1 A$. Note that $\J^U$ is based on the classical propositional logic. Although axioms and rules of Milnikel's logic is plausible, we think that the true logic of uncertain justifications is based on fuzzy logics. This is the approach we take in the following sections. Particularly, in Section \ref{sec:Fuzzy justification logics with truth constants} we will develop a fuzzy justification logic whose language contains truth constants, and show that all the principles of $\J^U$ is valid in it.
\section{T-norm based fuzzy logics}\label{sec:T-norm based fuzzy logics}
In this section we recall the Hajek's fuzzy basic logic {\sf BL} from \cite{Hajek1998}.\footnote{For a more detailed exposition of other fuzzy logics consult \cite{MOG2009}.} Formulas of {\sf BL} are built by the following grammar:
\[ A::= p~|~\bar{0}~|~A\& A~|~A\rightarrow A,\]
where $p\in\mathcal{P}$, $\&$ is the strong conjunction, and $\bar{0}$ is the truth constant for falsity. Further connectives are defined as follows:
\begin{eqnarray*}
  \neg A &:=& A\r \bar{0} \\  
  A\wedge B &:=& A\& (A\r B) \\
  A\vee B &:=& ((A\r B)\r B)\wedge((B\r A)\r A) \\
  A \equiv B &:=& (A\r B)\& (B\r A) \\
  A \leftrightarrow B &:=& (A\r B)\wedge(B\r A) 
 \end{eqnarray*}

\begin{definition}\label{def:axioms BL}
Axiom schemes of {\sf BL} are:
\begin{description}
  \item[BL1.] $(A\r B)\r((B\r C)\r(A\r C))$
  \item[BL2.] $(A\& B)\r A$
  \item[BL3.] $(A\& B)\r(B\& A)$
  \item[BL4.] $(A\&(A\r B))\r(B\&(B\r A))$
  \item[BL5a.] $(A\r(B\r C))\r((A\& B)\r C)$
  \item[BL5b.] $((A\& B)\r C)\r(A\r(B\r C))$
  \item[BL6.] $((A\r B)\r C)\r(((B\r A)\r C)\r C)$
  \item[BL7.] $\bar{0}\r A$
\end{description}
Modus Ponens is the only inference rule.
\end{definition}

We list here some theorems of {\sf BL} that are useful in later sections (for proofs cf. \cite{Hajek1998}).

\begin{lemma}\label{lem:theorems of BL}
The following are theorems of {\sf BL}:
\begin{enumerate}
\item $\bar{1}$.
\item $A \r (B\r A)$.
\item $(A\& B )\r (A \wedge B)$.
\item $A\wedge B \r A$.
\item $(A\r B) \r (A \r (A\wedge B))$.
\item $(A\r(B\r C))\r (B\r (A\r C))$.
\item $((A_1 \r B_1) \& (A_2 \r B_2)) \r ((A_1 \& A_2) \r (B_1 \& B_2))$.
\item $(A\r B)\vee (B\r A)$.
\end{enumerate}
\end{lemma}

The definition of the classical valuations can easily be extended to a many-valued realm. The standard set of truth degreees is the real interval $[0,1]$.  Now a truth valuation is a mapping $\V:\mathcal{P} \rightarrow [0,1]$. In order to define the truth functions of the strong conjunction $\&$ and implication $\r$ we need the notion of triangular norm (or t-norm).
A binary operation $*$ on the interval [0,1] is a t-norm if it is commutative, associative, non-decreasing and 1 is its unit element. 
\begin{definition}
A \textit{t-norm } is a binary operation $*: [0,1]^2 \r
[0,1]$ satisfying the following conditions. For all $x,y,z \in [0,1]$:
\begin{itemize}
    \item $x*y=y*x$, 
    \item $(x*y)*z=x*(y*z)$,
    \item $x\leq y$ implies $x*z \leq y*z$
    \item $1*x=x$.
\end{itemize}
$*$ is a \textit{continuous} t-norm if it is a t-norm and is a
continuous mapping of $[0,1]^2$ into $[0,1]$.
\end{definition}

\begin{example}\label{example: well-known t-norms}
Well-known examples of continuous t-norms are the followings (where $x,y \in [0,1]$):
\begin{enumerate}
    \item \textit{{\L}ukasiewicz t-norm}: $x*_L y=\max(0,x+y-1)$
    \item \textit{G\"{o}del t-norm}: $x*_G y=\min(x,y)$
    \item \textit{Product t-norm}: $x*_P y=x\cdot y$
\end{enumerate}
\end{example}
It is well-known that (see eg. \cite{Hajek1998}) each continuous t-norm is a combination of {\L}ukasiewicz, G\"{o}del and product t-norms (in a sense).

\begin{definition}
A residuum $\Rightarrow$ of a continuous t-norm $*$ is defined as follows:\footnote{In fact this is well-defined only if the t-norm $*$ is left-continuous.}
 \[x \Rightarrow y = \max\{z| x*z \leq y\},\]
 where $x,y \in [0,1]$.
\end{definition}

From the definition of residuum, for all $x,y,z \in [0,1]$, we have
\[ x*z \leq y ~~\mbox{if{f}}~~ z \leq x \Rightarrow y ~~\mbox{if{f}}~~ x \leq z\Rightarrow y.\] 

As it is shown in \cite{Hajek1998}, for each continuous t-norm $*$ and its residuum $\Rightarrow$, the following holds:
\[ x \Rightarrow y =1 ~\mbox{if{f}}~ x \leq y,\]
for all $x,y \in [0,1]$.

\begin{example}\label{example: well-known residuums}
The residuum of continuous t-norms of Example \ref{example: well-known t-norms} are the followings (where $x,y \in [0,1]$): 
\begin{enumerate}
    \item \textit{{\L}ukasiewicz implication}:  
    \[ x\Rightarrow_L y = \left\{ \begin{array}{ll}
                    1 & , x\leq y \\
                    1-x+y & , x> y
                    \end{array}
              \right.      \]
              or, $x \Rightarrow_L y= \min(1, 1-x+y)$.
    \item \textit{G\"{o}del implication}:  
    \[ x\Rightarrow_G y = \left\{ \begin{array}{ll}
                    1 & , x\leq y \\
                    y & , x> y
                    \end{array}
              \right.      \]
    \item \textit{Product (or Goguen) implication}: 
     \[ x\Rightarrow_P y = \left\{ \begin{array}{ll}
                    1 & , x\leq y \\
                    y/x & , x> y
                    \end{array}
              \right.      \]
\end{enumerate}
\end{example} 
Note that the only continuous implication of Example \ref{example: well-known residuums} is the {\L}ukasiewicz implication.

The following lemma is useful in later sections (we leave it to the reader to verify the details of the proof).
\begin{lemma}\label{lem:properties of L-implication}
For all $x,x',y,y'\in [0,1]$, we have:
\begin{enumerate}
\item If $x' \leq x$, then $x\Rightarrow_L y \leq x'\Rightarrow_L y$.
\item If $y \leq y'$, then $x\Rightarrow_L y \leq x\Rightarrow_L y'$.
\item $ (x \Rightarrow_L x') *_L (y \Rightarrow_L y')\leq (x*_L y) \Rightarrow_L (x' *_L y')$.
\end{enumerate}
\end{lemma}

Now suppose $*$ is an arbitrary t-norm. One can construct a t-norm based\footnote{A many valued logic is called \textit{t-norm based} if the truth function of all its connectives are defined from a particular t-norm, using possibly some truth constants (see e.g. \cite{Gottwald2000}).} propositional calculus $PC(*)$ (in the same language of ${\sf BL}$) and interpret the strong conjunction $\&$ and implication $\r$ as t-norm $*$ and its residuum $\Rightarrow$.  Now a truth valuation $\V$ (relative to $*$) recan be extended to all formulas as follows:
\begin{eqnarray*}
\V(\bar{0}) &=& 0.\\
\V(A\r B)&=& \V(A)\Rightarrow \V(B).\\
\V(A\& B)&=&  \V(A)* \V(B).
\end{eqnarray*}

It can be shown that 
\begin{eqnarray*}
\V(A\wedge B)&=& \min(\V(A), \V(B)).\\
\V(A\vee B)&=&  \max(\V(A), \V(B)).
\end{eqnarray*}

\begin{definition}
A formula $A$ is a 1-tautology of $PC(*)$ if{f} $\V(A)=1$ for each truth valuation $\V$ of $PC(*)$. A formula $A$ is a t-tautology if{f} $\V(A)=1$ for each truth valuation $\V$ and each continuous t-norm $*$.
\end{definition}

It is shown in \cite{Hajek1998} that every theorem of {\sf BL} is a t-tautology, the converse (completeness of {\sf BL}) is shown in \cite{CEGT2000}. Thus {\sf BL} is a common base of all the logics $PC(*)$. 

\begin{theorem}\label{thm:completeness BL}
A formula $F$ is provable in {\sf BL} if{f} it is a t-tautology.
 \end{theorem}
 
In the next definition, axiom systems of $PC(*)$ for t-norms of Example \ref{example: well-known t-norms} are given (cf. \cite{Hajek2010}).

\begin{definition}\label{def:axiomatic L-G-Pi}
From the basic logic {\sf BL} (in the same language of {\sf BL}) one gets a complete axiomatization of: 
\begin{enumerate}
\item the {\L}ukasiewicz logic {\sf {\L}} if one adds the axiom scheme:
\begin{equation}\label{eq:Lukasiewicz axiom}
\neg\neg A \r A \tag{L}
\end{equation}
\item the G\"{o}del  logic {\sf G} if one adds the axiom scheme 
\begin{equation}\label{eq:Godel axiom}
A\r (A\& A) \tag{G}
\end{equation}
\item  the product logic ${\sf \Pi}$ if one adds the axiom scheme
\begin{equation}\label{eq:Product axiom}
\neg\neg A \r ((A \r (A \& B)) \r (B \& \neg\neg B)) \tag{P}
\end{equation}
\end{enumerate}
\end{definition}

\begin{definition}
Let {\sf L} denote either ${\sf {\L}}$, {\sf G}, or ${\sf \Pi}$. A formula $A$ is a 1-tautology of ${\sf L}$ if{f} $\V(A)=1$ for each truth valuation $\V$ of ${\sf L}$. 
\end{definition}

\begin{theorem}\label{thm:completeness L-G-Pi}
Let {\sf L} denote either ${\sf {\L}}$, {\sf G}, or ${\sf \Pi}$. A formula $F$ is provable in ${\sf L}$ if{f} it is a 1-tautology of ${\sf L}$.
 \end{theorem}

Note that G\"{o}del  logic {\sf G} proves $A\& B \equiv A\wedge B$. Further,  {\sf G} is an extension of intuitionistic logic by axiom BL6, or by the linearity axiom $(A\r B)\vee(B\r A)$. Moreover, classical propositional logic is equivalent to either ${\sf {\L}}\cup{\sf G}$, or ${\sf \Pi}\cup{\sf G}$, or ${\sf {\L}}\cup{\sf \Pi}$ (cf. \cite{Hajek1998}).
\section{T-norm based justification logics}\label{sec:T-norm based justification logics}
In this section we introduce t-norm based justification logics. First we consider the basic justification logic {\sf J} on the bases of fuzzy basic logic {\sf BL} (instead  of classical propositional logic). We call the resulting system {\sf BLJ}. Then by adding axioms \ref{eq:Lukasiewicz axiom}, \ref{eq:Godel axiom}, and \ref{eq:Product axiom} from Definition \ref{def:axiomatic L-G-Pi} we obtain extensions of {\L}ukasiewicz, G\"{o}del and product logic. 

The language of {\sf BLJ} is an extension of the language of {\sf BL} by justification assertions $t:A$, where justification term $t$ is defined by the same grammar as in {\sf J}. Thus, formulas of {\sf BLJ} are built by the following grammar:
\[ A::= p~|~\bar{0}~|~A\& A~|~A\rightarrow A~|~t:A,\]
where $p\in\mathcal{P}$, and $t\in Tm$. Other connectives $\neg, \wedge, \vee, \equiv, \leftrightarrow$ are defined as in Section \ref{sec:T-norm based fuzzy logics}. $Fm_{\sf BLJ}$ denotes the set of all formulas of ${\sf BLJ}$.

{\sf BLJ} are axiomatized by adding axiom schemes Appl and Sum (from Definition \ref{def: axioms of J}) to axiom schemes of {\sf BL} (axioms BL1-BL7 from Definition \ref{def:axioms BL}). The rules of inference of {\sf BLJ} are the same as those of $\J$, i.e. MP and IAN (from Definition \ref{def: axioms of J}), with the difference that the axiom instance $A$ in the Iterated Axiom Necessitation rule should be now an axiom  instance of {\sf BLJ}.

A constant specification $\CS$ for {\sf BLJ} is a set of formulas of
the form $c_{i_n}:c_{i_{n-1}}:\ldots:c_{i_1}:A$ ($n\geq 1$), where $c_{i_j}$'s are
justification constants and $A$ is an axiom instance of {\sf BLJ}, such that it is downward closed. Let ${\sf BLJ}_\CS$ be the fragment of ${\sf BLJ}$ where
the Iterated Axiom Necessitation rule only produces formulas
from the given $\CS$. 

The definition of F-models for $\J$ can be easily extended to give models for {\sf BLJ}, by replacing the set of truth values $\{0,1\}$ to $[0,1]$. In fact, we define F-models for {\sf BLJ} in such a way that the truth valuation $\V$ and admissible evidence function $\E$ in many-valued logic {\sf BLJ} be a generalization of those in classical two-valued logic \J.

\begin{definition}\label{fuzzy M-model}
Let $*$ be a t-norm and $\Rightarrow$ be its residuum. A structure $\M_*=(\W,\R,\E,\V)$ (relative to $*$) is a fuzzy Fitting model for ${\sf BLJ}_\CS$ (or simply is a ${\sf BLJ}_\CS$-model), if $\W$ is a non-empty set of possible worlds, $\R$ is a (two-valued) accessibility relation on $\W$, $\V$ is a truth valuation of propositional variables in each world $\V:\W\times\mathcal{P}\r [0,1]$, and $\E$ is a fuzzy admissible evidence function $\E:\W\times Tm \times Fm_{\sf BLJ} \r [0,1]$ satisfying the following
conditions:\footnote{We continue to write $\E_w(t,A)$ for $\E(w,t,A)$, and $\V_w(A)$ for $\V(w,A)$.}
\begin{description}
 \item[$F\E 1.$] If $\E_w(s,A\r B)*\E_w(t,A)\leq\E_w(s\cdot t,B)$.
 \item[$F\E 2.$] $\E_w(s,A)\leq \E_w(t+s,A)$, $\E_w(s,A)\leq \E_w(s+t,A)$.
 \item[$F\E 3.$] If $c:F\in\CS$, then $\E_w(c,F)=1$, for every $w\in\W$.
\end{description}
The truth valuation $\V$ extends uniquely to all formulas of {\sf BLJ} as follows:
\begin{description}
\item[$\V1.$] $\V_w(\bar{0})=0$,
\item[$\V2.$] $\V_w(A\r B)= \V_w(A)\Rightarrow \V_w(B)$,
\item[$\V3.$] $\V_w(A\& B)= \V_w(A)* \V_w(B)$,
 \item[$\V4.$] $\V_w(t:A)=\E_w(t,A) * \V_w^{\Box}(A)$,
\end{description}
where
\[\V_w^\Box(A) = \inf\{ \V_v(A)~|~ w\R v\}.\]
A formula $A$ is ${\sf BLJ}_\CS$-valid if $\V_w(A)=1$, for every continuous t-norm $*$ and every ${\sf BLJ}_\CS$-model $\M_*=(\W,\R,\E,\V)$ and every $w\in\W$.
\end{definition}

Note that conditions $F\E1-F\E3$ in the above definition are generalizations of conditions $\E1-\E3$ of Definition \ref{def:F-model J}. The intended meaning of fuzzy admissible evidence function $\E$ is as follows: $t$ is an admissible evidence for $A$ in $w$ with the certainty degree $\E_w(t,A)$. Moreover, $\E_w(t,A)=1$ means that $t$ is an admissible evidence for $A$ in $w$ (as in the classical case). It seems that conditions $F\E 1$-$F\E 3$ are meaningful. For example, condition $F\E 2$ means that the certainty degree of $t+s$ (or $s+t$) for $A$ in $w$ is greater or at least equal to that of $t$ for $A$ in $w$. In other words, if we strengthen a justification (or reason) of a fact by adding another justification, then the certainty degree of our justification will be increased. 

The definition  of $\V_w^\Box$ indeed comes from the evaluation of modal formulas in many-valued modal logics with crisp accessibility relation (see e.g. \cite{Fitting1991,Priest2008})
\[\V_w(\Box A) = \inf\{ \V_v(A)~|~ w\R v\}.\]

Note that in many-valued modal logics it is not difficult to show that (see e.g. \cite{Priest2008})
\begin{equation}\label{eq:fuzzy modal e(K-axiom)}
\V_w(\Box(A\r B)) * \V_w(\Box A) \leq \V_w(\Box B)
\end{equation}
for every world $w$. This shows that the modal axiom K, $\Box(A\r B)\r(\Box A \r\Box B)$, is valid in many-valued modal logics with crisp accessibility relation. It is shown in \cite{BouEstevaGodo2007} that in general the axiom K is not valid in many-valued modal logics.

Next, we establish the soundness theorem for {\sf BLJ}.

\begin{theorem}\label{thm:soundness BLJ}
For any constant specification $\CS$ for {\sf BLJ}, if $F$ is provable in ${\sf BLJ}_\CS$, then $F$ is ${\sf BLJ}_\CS$-valid.
\end{theorem}
\begin{proof}
Assume that $F$ is provable in ${\sf BLJ}_\CS$. By induction on the proof of $F$ we show that $F$ is ${\sf BLJ}_\CS$-valid. Let $*$ be an arbitrary t-norm and $\M_*=(\W,\R,\E,\V)$ be an arbitrary ${\sf BLJ}_\CS$-model relative to $*$ and $w$ be  an arbitrary world in $\W$. We have to show that $\V_w(F)=1$. First suppose that $F$ is an axiom.
\begin{description}
  \item[$\bullet$] By Theorem \ref{thm:completeness BL}, all axioms of {\sf BL} are t-tautologies. Thus, the result is known for axioms of {\sf BL}.
  \item[$\bullet$] Suppose $F$ is $s:(A\r B)\r(t:A\r (s\cdot t):B)$. By $F\E 1$, we have 
  $$\E_w(s,A\r B)*\E_w(t,A)\leq\E_w(s\cdot t,B).$$
  On the other hand, from (\ref{eq:fuzzy modal e(K-axiom)}) we have
  \[\V_w^\Box(A\r B) * \V_w^\Box(A) \leq \V_w^\Box(B).\]
  Hence, by the properties of t-norms
  \[\E_w(s,A\r B)*\V_w^\Box(A\r B)*\E_w(t,A) * \V_w^\Box(A) \leq \E_w(s\cdot t,B) * \V_w^\Box(B).\]
     Thus, by $\V4$,
   \[ \V_w(s:(A\r B))* \V_w(t:A)\leq \V_w(s\cdot t:B),\]
   or
   \[ \V_w(s:(A\r B)\leq \V_w(t:A) \Rightarrow \V_w(s\cdot t:B),\]
   that is
   \[ \V_w(s:(A\r B)\leq \V_w(t:A \r s\cdot t:B).\]
   Therefore, $\V(F)=1$.
    \item[$\bullet$] Suppose $F$ is $s:A\r (s+t):A$. By $F\E2$, 
$$\E_w(s,A)\leq \E_w(s+t,A).$$
Hence, by the properties of t-norms
   $$\E_w(s,A) * \V_w^\Box(A)\leq \E_w(s+t,A) * \V_w^\Box(A).$$
Thus, by $\V4$, 
$$\V_w(s:A)\leq \V_w(s+t:A).$$
  Therefore, $\V_w(F)=1$.  The case for $s:A\r (t+s):A$ is similar.
    \end{description}
Now suppose $F$ is obtained from $G$ and $G\r F$ by MP. It is easy to verify that if $\V_w(G)=\V_w(G\r F)=1$ then $\V_w(F)=1$. 

Finally, suppose $F=c_{i_n}:c_{i_{n-1}}:\ldots:c_{i_1}:A$ is obtained by the rule IAN. Thus $c_{i_n}:c_{i_{n-1}}:\ldots:c_{i_1}:A\in\CS$. Since $\CS$ is downward closed we have 
\[c_{i_j}:c_{i_{j-1}}:\ldots:c_{i_1}:A\in\CS,\]
for every $1\leq j\leq n$. Since $\M$ is a ${\sf BLJ}_\CS$-model, we have $$\E_w(c_{i_j},c_{i_{j-1}}:\ldots:c_{i_1}:A)=1,$$
for every $1\leq j\leq n$. It is easy to prove, by induction on $j$, that 
$$\V_v(c_{i_j}:c_{i_{j-1}}:\ldots:c_{i_1}:A)=1,$$
for every $1\leq j\leq n$ and for every $v\in\W$. Therefore $$\V_w(c_{i_n}:c_{i_{n-1}}:\ldots:c_{i_1}:A)=1.$$\qed 
\end{proof}
 It remains open whether the converse of the above theorem, that is the completeness theorem, does hold. Let us now introduce Mkrtychev models (or M-models for short) for {\sf BLJ} which are singleton fuzzy Fitting models.

\begin{definition}\label{fuzzy M-model}
Let $*$ be a t-norm and $\Rightarrow$ be its residuum. A structure $M_*=(\E,\V)$ (relative to $*$) is a fuzzy M-model for ${\sf BLJ}_\CS$ (or ${\sf BLJ}_\CS$-M-model), if $\V$ is a truth valuation of propositional variables $\V:\mathcal{P}\r [0,1]$, and $\E$ is a fuzzy admissible evidence function $\E:Tm \times Fm_{\sf BLJ} \r [0,1]$ satisfying the following
conditions:
\begin{enumerate}
 \item If $\E(s,A\r B)*\E(t,A)\leq\E(s\cdot t,B)$.
 \item $\E(s,A)\leq \E(t+s,A)$, $\E(s,A)\leq \E(s+t,A)$.
 \item If $c:F\in\CS$, then $\E(c,F)=1$.
\end{enumerate}
The truth valuation $\V$ extends uniquely to all formulas of {\sf BLJ} as follows:
\begin{enumerate}
\item $\V(\bar{0})=0$,
\item $\V(A\r B)= \V(A)\Rightarrow \V(B)$,
\item $\V(A\& B)= \V(A)* \V(B)$,
 \item $\V(t:A)=\E(t,A)$.
\end{enumerate}
A formula $A$ is ${\sf BLJ}_\CS$-M-valid if $\V(A)=1$ for every continuous t-norm $*$ and every ${\sf BLJ}_\CS$-M-model $\M_*=(\E,\V)$.
\end{definition}

The soundness theorem of {\sf BLJ} with respect to M-models is a consequence of Theorem \ref{thm:soundness BLJ}.

\begin{theorem}\label{thm:soundness BLJ M-model}
For any constant specification $\CS$ for {\sf BLJ}, if $F$ is provable in ${\sf BLJ}_\CS$, then $F$ is ${\sf BLJ}_\CS$-M-valid.
\end{theorem}
Now let us define extensions of {\sf BLJ}.
\begin{definition}\label{def:axiomatic L-G-Pi-J}
The following logics have the same language as {\sf BLJ}.
\begin{enumerate}
\item {\sf {\L}J} is obtained from {\sf BLJ} by adding the axiom scheme \ref{eq:Lukasiewicz axiom}.
\item {\sf GJ} is obtained from {\sf BLJ} by adding the axiom scheme \ref{eq:Godel axiom}.
\item  ${\sf \Pi J}$ is obtained from {\sf BLJ} by adding the axiom scheme \ref{eq:Product axiom}.
\end{enumerate}
\end{definition}

In the rest of this section let ${\sf L}$ denote either {\sf {\L}J}, {\sf GJ}, or ${\sf \Pi J}$. Note that the axiom instance $A$ in the Iterated Axiom Necessitation rule in {\sf L} should be now an axiom  instance of {\sf L}. The definition of constant specifications $\CS$ for {\sf L} and ${\sf L}_\CS$ are similar to those of \J. 

\begin{definition}
An  ${\sf BLJ}_\CS$-model $\M_*=(\W,\R,\E,\V)$ is an 
\begin{enumerate}
\item {\sf {\L}J}-model if the t-norm $*$ is $*_L$.
\item {\sf GJ}-model if the t-norm $*$ is $*_G$.
\item ${\sf \Pi J}$-model if the t-norm $*$ is $*_P$.
\end{enumerate}
A formula $A$ is ${\sf L}_\CS$-valid if $\V_w(A)=1$ for each ${\sf L}_\CS$-model $\M=(\W,\R,\E,\V)$ and every $w\in\W$.
\end{definition}

Now soundness of {\sf {\L}J}, {\sf GJ}, and ${\sf \Pi J}$ with respect to their models can be obtained easily from Theorems \ref{thm:completeness L-G-Pi} and \ref{thm:soundness BLJ}. 

\begin{theorem}\label{thm:soundness LJ-GJ-PiJ}
For a given constant specification $\CS$ for {\sf L}, if $F$ is provable in ${\sf L}_\CS$, then $F$ is ${\sf L}_\CS$-valid.
\end{theorem}

M-models for {\sf {\L}J}, {\sf GJ}, and ${\sf \Pi J}$ can be defined similar to M-models of {\sf BLJ}, and the soundness theorem also holds for M-validity.
\section{Fuzzy justification logics with truth constants}\label{sec:Fuzzy justification logics with truth constants}
In classical propositional logic the truth constants $\top$ (truth) and $\bot$ (falsity) can be defined respectively by $p\vee\neg p$ and $p\wedge\neg p$, for some propositional variable $p$. But in many-valued logics (introduced in Section \ref{sec:T-norm based justification logics}) we cannot define all truth values $[0,1]$ in our object language, and thus they may be added explicitly to the language. A well-known fuzzy logic with truth constants was introduced by Pavelka \cite{Pavelka} (for a recent account of this logic see \cite{Novak2000}). Hajek \cite{Hajek1998} gives a simple formulation of this logic, called \textit{Rational Pavelka Logic} {\sf RPL}, in which rational numbers in $[0,1]$, as truth constants,  are added to the language of {\L}ukasiewicz logic {\sf {\L}}. In this section, we first recall {\sf RPL}, and then extend it with justification assertions, and prove a soundness and a graded-style completeness theorem.

The language of {\sf RPL} is an extension of the language of {\L}ukasiewicz logic {\sf {\L}} by truth constants $\bar{r}$, for each rational number $r\in\mathbb{Q}\cap[0,1]$. More precisely, formulas of {\sf RPL} is constructed by the following grammar:
\[ A::= p~|~\bar{r}~|~A\& A~|~A\rightarrow A,\]
where $p\in\mathcal{P}$ and $r\in\mathbb{Q}\cap[0,1]$. Formulas of the form $\bar{r} \r A$ are called \textit{graded formulas}. Axioms and rules of  {\sf RPL} is given by adding the following bookkeeping axiom schemes for truth constants to axioms and rules of  {\sf{\L}}:
 
\begin{description}
  \item[TC1.] $(\bar{r}\r \bar{r'}) \equiv \overline{r\Rightarrow_L r'}$,
   \item[TC2.] $(\bar{r}\& \bar{r'}) \equiv \overline{r*_L r'}$,
  \end{description}
for all $r,r'\in\mathbb{Q}\cap[0,1]$.

Note that if $r\leq r'$, then $r \Rightarrow_L r' = 1$. Therefore, by TC1, $(\bar{r}\r \bar{r'}) \equiv \overline{r\Rightarrow_L r'} =\bar{1}$. Since $\RPL\vdash \bar{1}$, we can conclude that $\RPL\vdash \bar{r} \r \bar{r'}$.

A truth valuation $\V$ of formulas of {\sf RPL} is defined similar to that of {\sf {\L}}, with the following addition:
\[ \V(\bar{r})=r\]
for any $r\in\mathbb{Q}\cap[0,1]$. By this definition, it is obvious that a graded formula $\bar{r} \r A$ is true if{f} the truth value of $A$ is at least $r$, i.e. 
\[ \V(\bar{r} \r A)=1 ~~\mbox{if{f}}~~ r\leq \V(A).\]

A truth valuation $\V$ of $\RPL$ is a model of a finite set of $\RPL$-formulas $T$, if $\V(A)=1$ for all $A\in T$. The following strong completeness theorem for $\RPL$ is proved in \cite{Hajek1998}.

\begin{theorem}\label{thm:strong completeness RPL}
Let $T$ be a finite set of $\RPL$-formulas, and $F$ be an $\RPL$-formula. $\RPL\vdash F$ if{f} $F$ is true in every model of $T$.
\end{theorem}

Now by adding justification assertions to {\sf RPL}, we give a Pavelka based justification logic. We denote this system by {\sf RPLJ}. Formulas of $\F$ are constructed by the following grammar:
\[ A::= p~|~\bar{r}~|~A\& A~|~A\rightarrow A~|~t:A,\]
where $p\in\mathcal{P}$, $r\in\mathbb{Q}\cap[0,1]$, and $t\in Tm$ (here $Tm$ is the same set of justification terms of basic justification logic $\J$ in Section \ref{sec:Basic justification logic}). The connectives $\neg, \wedge, \vee, \equiv, \leftrightarrow$ are defined as in {\sf BL}. 

Thanks to truth constants, the following useful operators can be defined in the language of $\F$:
\begin{eqnarray*}
  t:_r A &:=& \bar{r}\r t:A \\
  t:^r A &:=&  t:A \r \bar{r} \\
  t \overset{r}{:} A &:=& \bar{r}\leftrightarrow t:A = (t:_r A) \wedge (t:^r A)
\end{eqnarray*}

Informally, $t:_r A$, $t:^r A$, and $t \overset{r}{:} A$ can be read respectively as (see Lemma \ref{lem:truth value graded justification assertions})
\begin{itemize}
\item ``$t$ is a justification for believing $A$ with certainty degree at least $r$",
\item ``$t$ is a justification for believing $A$ with certainty degree at most $r$", 
\item  ``$t$ is a justification for believing $A$ with  certainty degree $r$."
\end{itemize}

Formulas of the form $t:_r A$, $t:^r A$, or $t \overset{r}{:} A$ are called \textit{graded justification assertions}.\footnote{A similar approach but in the framework of modal logics was given by  Mironov \cite{Mironov2005}. Mironov presented a fuzzy modal logic with modalities $\Box_a$, where $a$ is a member of a complete lattice, so that $\Box_a A$ is read ``the plausibility measure of $A$ is equal to $a$."}

\begin{definition}
Pavelka based justification logic \F~is given by adding the axiom schemes Sum and Appl (from Definition \ref{def: axioms of J}) and the following rule to axioms and rules of \RPL: 
\begin{description}
 \item[GIAN.] Graded Iterated Axiom Necessitation Rule: $$\vd c_{i_n}\overset{1}{:} c_{i_{n-1}}\overset{1}{:} \ldots\overset{1}{:} c_{i_1}\overset{1}{:} A,$$ where $A$ is an axiom instance of \F, $c_{i_j}$'s are arbitrary justification constants and $n\geq 1$.
\end{description}
\end{definition}

Note that in contrast to the language of {\sf BLJ}, in $\F$ we are able to express explicitly the certainty degree of a justification for a statement. Particularly, in {\sf BLJ} for an axiom instance $A$, by rule IAN, we could deduce $c:A$, for some constant $c$, which means that $c$ is an evidence for $A$ with a certainty degree taken form the interval $[0,1]$. But in $\F$, we could deduce $c\overset{1}{:} A$ which means that $c$ is a completely convincing evidence for $A$, and this is what we expect.

\begin{definition}[Constant specification]
A \textit{constant specification} $\CS$ is a set of formulas of
the form $$c_{i_n}\overset{1}{:} c_{i_{n-1}}\overset{1}{:} \ldots\overset{1}{:} c_{i_1}\overset{1}{:} A,$$ where $n\geq 1$, $c_{i_j}$'s are arbitrary justification constants and $A$ is an axiom instance of \F, such that it is downward closed..
\end{definition}

\begin{definition}
A constant specification $\CS$ is \textit{axiomatically appropriate} if for each axiom instance $A$ of \F~there is a constant $c$ such that $c\overset{1}{:}A\in\CS$, and in addition if $F\in\CS$, then $c\overset{1}{:}F\in\CS$ for some constant $c$.
\end{definition}
For a constant specification $\CS$, $\F_\CS$ is a fragment of \F~in which the rule GIAN only produces formulas from the given $\CS$. For a set of formulas $T$, formula $A$ and constant specification $\CS$, we write $T\vd_\CS A$ if $A$ is derived from  $T$ in $\F_\CS$. 

\begin{lemma}\label{lem:theorems of RPLJ}
Given an arbitrary constant specification $\CS$, the following statements are theorems of $\F_\CS$. For all \F-formulas $A$, all $t\in Tm$, and all $r,r'\in \mathbb{Q}\cap[0,1]$:
\begin{enumerate}
\item $t:^1 A$.
\item $t:_0 A$.
\item $\neg t:^r A \r t:_r A$.
 \item $\neg t:_r A \r t:^r A$.
   \item $t:_{r'} A \r t:_r A$, where $r \leq r'$.
      \item $t:_1 A\equiv t\overset{1}{:} A$.
      \item $t\overset{1}{:} A \r t:A$.
      \item $t:_r A \vee t:^r A$.
\end{enumerate}
\end{lemma}
\begin{proof}
\begin{enumerate}
\item By Lemma \ref{lem:theorems of BL} clause 2, $$\bar{1} \r (t:A \r \bar{1}).$$ Since $\bar{1}$ is provable in  $\F_\CS$ (Lemma \ref{lem:theorems of BL} clause 1), by MP, $$ t:A \r \bar{1}.$$ That is, $ t:^1 A$.
\item $t:_0 A$ is an instance of axiom BL7, i.e. $\bar{0} \r t:A$.
\item Note that $\neg t:^r A$ is an abbreviation of $(t:A \r \bar{r}) \r \bar{0}$. Now by axiom BL6,
\[((t:A \r \bar{r}) \r \bar{0}) \r (((\bar{r} \r t:A) \r \bar{0} ) \r \bar{0}),\]
or, in our notations, $\neg t:^r A \r \neg \neg t:_r A$. Then, by axiom \ref{eq:Lukasiewicz axiom} of {\L}ukasiewicz logic and axiom BL1 using MP, we obtain $\neg t:^r A \r t:_r A$.
\item Similar to clause 3.
\item As mentioned above, if $r \leq r'$, then $\vdash_\CS \bar{r} \r \bar{r'}$, and therefore by axiom BL1 and MP, we obtain $$(\bar{r'} \r t:A) \r (\bar{r} \r t:A).$$ That is, $t:_{r'} A \r t:_r A$.
\item By Lemma \ref{lem:theorems of BL} clause 4, $t\overset{1}{:} A \r t:_1 A$. For the converse, by Lemma \ref{lem:theorems of BL} clause 2, we have $$t:^1 A \r (t:_1 A \r t:^1 A).$$ From this and clause 1 using MP, we obtain $t:_1 A \r t:^1 A$. Hence, by Lemma \ref{lem:theorems of BL} clause 5 using MP, $$t:_1 A \r (t:_1 A \wedge t:^1 A).$$ That is, $t:_1 A\r t\overset{1}{:} A$.
\item By clause 6, $t\overset{1}{:} A \r t:_1 A$, or $t\overset{1}{:} A \r (\bar{1}\r t:A)$ is provable in  $\F_\CS$. By Lemma \ref{lem:theorems of BL} clause 6, 
\[( t\overset{1}{:} A \r (\bar{1}\r t:A))\r (\bar{1} \r (t\overset{1}{:} A \r t:A))\]
is provable in  $\F_\CS$, from this by MP we infer $\bar{1} \r (t\overset{1}{:} A \r t:A)$. And thus, $t\overset{1}{:} A \r t:A$  is provable in  $\F_\CS$.
\item By Lemma \ref{lem:theorems of BL} clause 8, $(\bar{r}\r t:A)\vee(t:A \r \bar{r})$. That is $t:_r A \vee t:^r A$.
\qed
\end{enumerate} 
\end{proof}

\begin{lemma}
Given an arbitrary constant specification $\CS$, the following rules are admissible in $\F_\CS$. For all \F-formulas $A$ and $B$, all $t\in Tm$, and all $r,r'\in \mathbb{Q}\cap[0,1]$:
\begin{enumerate}
\item Graded Modus Ponens:
 $$\frac{\bar{r}\r (A\r B) \quad\bar{r'}\r A}{\overline{r*_L r'}\r B} \ GMP$$
 \item  Justified Graded Modus Ponens:
 $$\frac{\bar{r}\r s:(A\r B) \quad\bar{r'}\r t:A}{\overline{r*_L r'}\r s\cdot t:B}\ JGMP$$
or 
$$\frac{s:_r (A\r B) \quad t:_{r'} A}{s\cdot t:_{ r*_L r'} B}\ JGMP$$
\item Monotonicity:
$$\frac{s:_r A}{s+t:_r A}\ Mon_1\qquad\frac{s:_r A}{t+s:_r A}\ Mon_2$$
 \end{enumerate}
\end{lemma}
\begin{proof}
The proof of clause 1 can be found in \cite{Hajek1998}.  For clause 2, suppose $\bar{r}\r s:(A\r B)$ and $\bar{r'}\r t:A$ are provable in $\F_\CS$. Thus, $$(\bar{r}\r s:(A\r B))\&(\bar{r'}\r t:A)$$ is provable in $\F_\CS$. Hence, from Lemma \ref{lem:theorems of BL} clause 7 using MP, $$(\bar{r}\& \bar{r'}) \r (s:(A\r B)\& t:A).$$
From this and axiom TC2, it follows that 
$$\overline{r*_L r'} \r (s:(A\r B)\& t:A).$$
Moreover, from axiom Appl and BL5a we deduce that 
\[(s:(A\r B)\& t:A)\r s\cdot t:B.\]
Therefore,
$\overline{r*_L r'} \r s\cdot t:B$. 

For clause 3, suppose that $s:_r A$, i.e. $\bar{r}\r s:A$, is provable in in $\F_\CS$. Now from this and axiom BL1, we obtain $\bar{r}\r s+t:A$, or $s+t:_r A$.\qed
\end{proof}

From rules JGMP, Mon$_1$, and Mon$_2$, it is immediately follows that the following rules are also admissible in $\F_\CS$:
 $$\frac{s\overset{r}{:} (A\r B) \quad t\overset{r'}{:} A}{s\cdot t:_{r*_L r'} B}\ JGMP',\qquad\frac{s\overset{r}{:} A}{s+t:_r A}\ Mon'_1,\qquad\frac{s\overset{r}{:} A}{t+s:_r A}\ Mon'_2.$$
 
The following form of the Deduction Theorem holds in \F. The proof is similar to that given in \cite{Hajek1998} for {\sf BL}. 
\begin{lemma}[Deduction Theorem]\label{lemma:Deduction Theorem}
$T,A\vdash_\CS B$ iff $T\vdash_\CS A^n\r B$ for some $n$, where $A^n:=A\&\ldots\& A$, $n$ times.
\end{lemma}

Note that the ordinary deduction theorem, $T,A\vdash B$ iff $T\vdash A\r B$, only holds in G\"{o}del justification logic {\sf GJ} (for more details cf. \cite{Hajek1998}).

\begin{lemma}[Lifting Lemma]\label{lem:lifting lemma RPLJ}
Given axiomatically appropriate constant specification $\CS$, if $$A_1,\ldots,A_n \vdash_\CS F,$$ then there is a justification term
$t(x_1,\ldots,x_n)$, for variables $x_1,\ldots,x_n$, such that 
$$ x_1\overset{1}{:} A_1,\ldots,x_n\overset{1}{:} A_n \vdash_\CS  t(x_1,\ldots,x_n)\overset{1}{:} F.$$
\end{lemma}
\begin{proof}
The proof is by induction on the derivation of
$F$. We have two base cases:
\begin{itemize}
 \item If $F$ is an axiom, then there is a justification constant $c$ such that $c:F\in\CS$. Put $t:= c$, and use IAN to obtain $c\overset{1}{:} F$.
  \item If $F=A_i$, then put $t:=x_i$.
\end{itemize}
For the induction step we have two cases.
\begin{itemize}
 \item Suppose $F$ is obtained by Modus Ponens from $G\r F$ and $G$. By the induction hypothesis, there are terms $u(x_1,\ldots,x_n)$
 and $v(x_1,\ldots,x_n)$ such that 
$$  x_1\overset{1}{:} A_1,\ldots,x_n\overset{1}{:} A_n \vdash_\CS  u(x_1,\ldots,x_n)\overset{1}{:} (G\r F),$$
and
 $$  x_1\overset{1}{:} A_1,\ldots,x_n\overset{1}{:} A_n \vdash_\CS  v(x_1,\ldots,x_n)\overset{1}{:} G.$$
 
 Then put $t:=u.v$, and use the rule JGMP$'$ to obtain 
 $$ x_1\overset{1}{:} A_1,\ldots,x_n\overset{1}{:} A_n \vdash_\CS  u(x_1,\ldots,x_n)\cdot v(x_1,\ldots,x_n):_1 F,$$
 which, by Lemma \ref{lem:theorems of RPLJ} clause 6, implies that
 $$  x_1\overset{1}{:} A_1,\ldots,x_n\overset{1}{:} A_n \vdash_\CS  u(x_1,\ldots,x_n)\cdot v(x_1,\ldots,x_n)\overset{1}{:} F.$$
 
 \item Suppose $F$ is obtained from the rule GIAN, so $F=c_{i_n}\overset{1}{:} c_{i_{n-1}}\overset{1}{:} \ldots:c_{i_1}\overset{1}{:} B\in\CS$, for some axiom instance $B$. Then since $\CS$ is axiomatically appropriate, there is a justification constant $c$ such that $c\overset{1}{:} c_{i_n}\overset{1}{:} c_{i_{n-1}}\overset{1}{:} \ldots:c_{i_1}\overset{1}{:} B\in\CS$. Put $t:=c$.\qed
\end{itemize}
\end{proof}

\begin{lemma}[Internalization Lemma]\label{lem:Internalization Lemma RPLJ}
Given an axiomatically appropriate constant specification $\CS$, if $  \vdash_\CS F$, then there is a justification term
$t$ such that $ \vdash_\CS t:F$.
\end{lemma}
\begin{proof}
Special case of Lemma \ref{lem:lifting lemma RPLJ}.\qed
\end{proof}

Next we introduce fuzzy Fitting models for \F.
\begin{definition}
A structure $\M=(\W,\R,\E,\V)$ is a fuzzy Fitting model for $\F_\CS$ (or simply is a $\F_\CS$-model), if it is a {\sf {\L}J}-model such that  $\E$ is a fuzzy admissible evidence function $\E:\W\times Tm \times Fm_\F \r [0,1]$ satisfying the following conditions:
\begin{description}
 \item[$F\E 1.$] If $\E_w(s,A\r B)*_L \E_w(t,A)\leq\E_w(s\cdot t,B)$.
 \item[$F\E 2.$] $\E_w(s,A)\leq \E_w(t+s,A)$, $\E_w(s,A)\leq \E_w(s+t,A)$.
 \item[$F\E 3.$] If $c\overset{1}{:} F\in\CS$, then $\E_w(c,F)=1$, for every $w\in\W$.
\end{description}
The truth valuation $\V$ extends uniquely to all formulas of {\sf BLJ} as follows:
\begin{description}
\item[$\V1.$] $\V_w(\bar{r})=r$, for all $r\in \mathbb{Q}\cap[0,1]$,
\item[$\V2.$] $\V_w(A\r B)= \V_w(A)\Rightarrow_L \V_w(B)$,
\item[$\V3.$] $\V_w(A\& B)= \V_w(A)*_L \V_w(B)$,
 \item[$\V4.$] $\V_w(t:A)=\E_w(t,A) *_L \V_w^{\Box}(A)$,
\end{description}
where
\[\V_w^\Box(A) = \inf\{ \V_v(A)~|~ w\R v\}.\]
\end{definition}

\begin{definition}
\begin{itemize}
\item For a $\F_\CS$-model $\M=(\W,\R,\E,\V)$ and a set of formulas $T$, $\M$ is a model of $T$ if $\V_w(F)=1$, for all formulas $F\in T$ and all $w\in\W$.
\item We say that a formula $F$ is a semantic consequence of $T$ in $\F_\CS$, written $T\models_\CS  F$, if every $\F_\CS$-model of $T$ is a $\F_\CS$-model of $F$ (or more precisely of $\{F\}$).
\item A formula $F$ is $\F_\CS$-valid if $\models_\CS  F$.
 \end{itemize}
\end{definition}
For a model $\M=(\W,\R,\E,\V)$ sometimes we write $w\in\M$ instead of $w\in\W$.

 The truth valuation of graded justification assertions are as follows:
\begin{align}
 \V_w(t:_r A) &= r\Rightarrow_L \V_w(t:A),\label{eq:evaluation of graded formulas 1} \\
  \V_w(t:^r A) &=  \V_w(t:A) \Rightarrow_L r, \label{eq:evaluation of graded formulas 2}\\
   \V_w(t \overset{r}{:} A) &= \min(\V_w(t:_r A), \V_w(t:^r A)). \label{eq:evaluation of graded formulas 3}
\end{align}

\begin{lemma}\label{lem:truth value graded justification assertions}
Let $\M$ be an arbitrary $\F_\CS$-model and $w\in\M$ be an arbitrary world. Then, the following holds:
\begin{enumerate}
\item $\V_w(t:_r A)=1$ if{f} $\V_w(t:A)\geq r$.
\item $\V_w(t :^r A)=1$ if{f} $\V_w(t:A)\leq r$.
\item $\V_w(t \overset{r}{:} A)=1$ if{f} $\V_w(t:A)=r$.
\end{enumerate}
\end{lemma}
\begin{proof}
The proof of items 1 and 2 follows easily from (\ref{eq:evaluation of graded formulas 1}) and (\ref{eq:evaluation of graded formulas 2}). For item 3, we reason as follows:
\begin{eqnarray*}
  \V_w(t \overset{r}{:} A)=1 &~\mbox{if{f}}~& \min(\V_w(t:_r A), \V_w(t:^r A))=1\\
 &~\mbox{if{f}}~&  \V_w(t:_r A)= \V_w(t:^r A)=1\\
 &~\mbox{if{f}}~& r\leq \V_w(t:A) ~\mbox{and}~ \V_w(t:A) \leq r \\
  &~\mbox{if{f}}~& \V_w(t:A) = r. 
\end{eqnarray*}\qed
\end{proof}

\begin{theorem}[Soundness]\label{thm:soundness RPLJ}
For each set of formulas $T$, formula $A$ and constant specification $\CS$, if $T\vd_\CS  A$ then $T\models_\CS  A$.
\end{theorem}
\begin{proof}
The proof is by induction on the length of the proof of $A$. The case where $A\in T$ is obvious. The case where $A$ is an axiom or is obtained from MP is similar to the proof of Theorem \ref{thm:soundness BLJ} (in particular, it is easy to show that axioms TC1 and TC2 are $\F_\CS$-valid).

Suppose $F=c_{i_n}\overset{1}{:} c_{i_{n-1}}\overset{1}{:} \ldots\overset{1}{:} c_{i_1}\overset{1}{:} A$ is obtained by the rule GIAN. Thus $c_{i_n}\overset{1}{:} c_{i_{n-1}} \overset{1}{:}\ldots\overset{1}{:} c_{i_1}\overset{1}{:} A\in\CS$. Since $\CS$ is downward closed we have 
\[c_{i_j}\overset{1}{:} c_{i_{j-1}}\overset{1}{:} \ldots\overset{1}{:} c_{i_1}\overset{1}{:} A\in\CS,\]
for every $1\leq j\leq n$. Let $\M$ be a ${\sf BLJ}_\CS$-model, and $w\in\M$. From $F\E3$ we have $$\E_w(c_{i_j},c_{i_{j-1}}\overset{1}{:} \ldots\overset{1}{:} c_{i_1}\overset{1}{:} A)=1,$$
for every $1\leq j\leq n$. 

It is not difficult to prove, by induction on $j$, that 
$$\V_v(c_{i_j}\overset{1}{:} c_{i_{j-1}}\overset{1}{:} \ldots\overset{1}{:} c_{i_1}\overset{1}{:} A)=1,$$
for every $1\leq j\leq n$ and for every $v\in\W$. We only check the base case $j=1$. Suppose $c \overset{1}{:} A\in\CS$, and $v\in\W$. We want to show that $\V_v(c \overset{1}{:} A)=1$. It suffices to prove that $\V_v(c:_1 A)=\V_v(c :^1 A)=1$. Observe that
\[\V_v(c:A)=\E_v(c,A) *_L \V_v^\Box(A)= 1 *_L \V_v^\Box(A)= \V_v^\Box(A)= \inf\{\V_u(A)~|~v\R u\}.\]
Since $A$ is an axiom, $\V_u(A)=1$ for all $u\in\W$. Thus, $\V_v(c:A)=1$. Hence,
\[\V_v(c:_1 A)= 1 \Rightarrow_L \V_v(c:A)= 1 \Rightarrow_L 1=1,\]
\[\V_v(c:^1 A)=  \V_v(c:A)\Rightarrow_L 1= 1 \Rightarrow_L 1=1.\]
Thus, 
\[\V_v(c \overset{1}{:} A)= \min(\V_v(c:_1 A), \V_v(c :^1 A))=1.\]
This completes the base case. The induction step is routine.

Therefore, $\V_w(c_{i_n}\overset{1}{:} c_{i_{n-1}}\overset{1}{:} \ldots\overset{1}{:} c_{i_1}\overset{1}{:} A)=1.$
\qed
\end{proof}

Next we show that all principles of Milnikel's logic of uncertain justifications (\cite{Milnikel2014}) are valid in \F.
\begin{lemma}
Given an arbitrary constant specification $\CS$, the following statements are $\F_\CS$-valid:
\begin{enumerate}
\item $s:_r (A\r B)\r(t:_{r'} A \r s\cdot t:_{r *_L r'} B)$.
\item $s:_r A\r s+t:_r A$, $s:_r A\r t+s:_r A$.
\item $t:_{r'} A \r t:_r A$, where $r \leq r'$.
\end{enumerate}
\end{lemma}
\begin{proof}
Let $\M$ be an arbitrary $\F_\CS$-model and $w\in\M$ be an arbitrary world.
\begin{enumerate}
\item From the proof of soundness theorem we have
\[\V_w(s:(A\r B))*_L \V_w(t:A)\leq \V_w(s\cdot t:B).\]
By Lemma \ref{lem:properties of L-implication} (clause 2), 
\[(r *_L r') \Rightarrow_L (\V_w(s:(A\r B))*_L \V_w(t:A))\leq (r *_L r') \Rightarrow_L \V_w(s\cdot t:B).\]
By Lemma \ref{lem:properties of L-implication} (clause 3),
\[(r \Rightarrow_L \V_w(s:(A\r B))*_L (r'\Rightarrow_L \V_w(t:A))\leq (r *_L r') \Rightarrow_L \V_w(s\cdot t:B).\]
That is
\[\V_w(s:_r (A\r B))*_L \V_w(t:_{r'} A)\leq \V_w(s\cdot t:_{r*_L r'} B).\]
Hence,
\[\V_w(s:_r (A\r B))\leq \V_w(t:_{r'} A)\Rightarrow_L \V_w(s\cdot t:_{r*_L r'} B).\]
Thus, $$\V_w(s:_r (A\r B)\r(t:_{r'} A \r s\cdot t:_{r *_L r'} B))=1.$$
\item  From $F\E2$ we have 
$$\E_w(s,A)\leq \E_w(t+s,A),$$ 
$$\E_w(s,A)\leq \E_w(s+t,A).$$
From these and properties of t-norms we obtain
$$\E_w(s,A)*_L \V_w^{\Box}(A)\leq \E_w(t+s,A)*_L \V_w^{\Box}(A),$$ 
$$\E_w(s,A)*_L \V_w^{\Box}(A)\leq \E_w(s+t,A)*_L \V_w^{\Box}(A).$$
Thus,
$$\V_w(s:A)\leq \V_w(t+s:A),$$ 
$$\V_w(s:A)\leq \V_w(s+t:A).$$
By Lemma \ref{lem:properties of L-implication} (clause 2),
$$r \Rightarrow_L \V_w(s,A)\leq r \Rightarrow_L \V_w(t+s,A),$$ 
$$r \Rightarrow_L\V_w(s,A)\leq r \Rightarrow_L \V_w(s+t,A).$$
That are
$$\V_w(s:_r A)\leq \V_w(t+s:_r A),$$ 
$$\V_w(s:_r A)\leq \V_w(s+t:_r A).$$ 
\item By Lemma \ref{lem:theorems of RPLJ} (clause 5) and Theorem \ref{thm:soundness RPLJ}.\qed
\end{enumerate}
\end{proof}

\begin{definition}
Let $T$ be a set of formulas, and $A$ be a formula.
\begin{enumerate}
  \item Truth degreee of $A$ over $T$ respecting $\CS$ is defined as follows: 
  $$\| A\|^\CS_T :=\inf\{\V_w(A)| \M~\mbox{is a}~\F_\CS\mbox{-model of}~T~\mbox{and}~w\in\M\}.$$
  \item Provability degreee of $A$ over $T$ respecting $\CS$ is defined as follows: 
  $$| A|^\CS_T :=\sup\{r| T\vd_\CS  \bar{r}\r A\}.$$
\end{enumerate}
\end{definition}

\begin{definition}
Let $S$ be a set of formulas. $S$ is $\CS$-consistent if $S\not\vd_\CS\bar{0}$. $S$ is $\CS$-complete if $S\vd_\CS A\r B$ or $S\vd_\CS B\r A$, for each pair of formulas $A$, $B$.
\end{definition}

The proof of the following lemmas are similar to that given in \cite{Hajek1998} (Lemmas 3.3.7 and 3.3.8).

\begin{lemma}
If $S\not\vdash_\CS \bar{r} \r A$, then $S\cup\{ A\r \bar{r}\}$ is $\CS$-consistent.
\end {lemma}

\begin{lemma}\label{lem:properties of provability degreee}
Let $S$ be a \CS-consistent and \CS-complete set of formulas. The following properties hold:
\begin{enumerate}
   \item $|A\r B|^\CS_S=|A|^\CS_S \Rightarrow_L |B|^\CS_S$.\vspace{0.1cm}
  \item $|A\& B|^\CS_S=|A|^\CS_S *_L |B|^\CS_S$.
\end{enumerate}
\end{lemma}

Moreover, it can be shown that every $\CS$-consistent set of formulas can be extended to a \CS-consistent and \CS-complete set.

\begin{definition}[Canonical model]\label{def:canonical model}
The \textit{canonical model} $\M^c=(\W^c,\R^c,\E^c,\V^c)$ for $\F_\CS$ is defined as follows:
\begin{enumerate}
  \item $\W^c = \{ S|S~\mbox{is a $\CS$-consistent and $\CS$-complete set}\}$.\vspace{0.1cm}
  \item $S_1 \R^c S_2$ if{f} $S_1^\sharp=\{A|t:A\in S_1\} \subseteq S_2$, for $S_1,S_2\in\W^c$.\vspace{0.1cm}
    \item $\E_S^c(t,A):= |t:A|^\CS_S$, for $S\in\W^c$.\vspace{0.1cm}
    \item $\V_S^c(p):=|p|^\CS_S$, for $S\in\W^c$ and $p\in\mathcal{P}$.
\end{enumerate}
\end{definition}

\begin{lemma}[Truth Lemma]\label{lem:Truth Lemma}
Let $\M^c=(\W^c,\R^c,\E^c,\V^c)$ be the canonical model of $\F_\CS$. For every formula $F$  and every $S\in\W^c$, if $F\in S$, then
\[\V_S^c(F)=|F|^\CS_S = 1.\]
\end{lemma}
\begin{proof}
First note that for every formula $F$ in $S$, we have $S\vdash_\CS \bar{1}\r F$. Hence $|F|^\CS_S=1$. Next by induction on the complexity of $F$ we show that for every $S\in\W^c$ if $F\in S$, then 
\[\V_S^c(F)=|F|^\CS_S.\]
 The case when $F$ is a propositional variable $p$ follows from Definition \ref{def:canonical model}. If $F=\bar{r}$ and $S\in\W^c$, then $\V_S^c(\bar{r})=r$, and by definition of provability degreee $|\bar{r}|_S^\CS=r$. Thus $$\V_S^c(\bar{r})=|\bar{r}|_S^\CS=r.$$ 
The cases for which the main connective of $F$ is $\&$ or $\r$ follows from Lemma \ref{lem:properties of provability degreee}. Finally, for $S\in\W^c$ if $F=t:A$ is in $S$, then $|t:A|^\CS_S =1$. Furthermore, for any $S'\in\W$ such that $S \R^c S'$, by definition of $\R^c$, we have $A\in S'$. By the induction hypothesis, $\V_{S'}^c(A)=|A|_{S'}^\CS =1$. Thus, 
 \[\V_S^{c\Box}(A) = \inf \{ \V_{S'}^c(A)~|~ S\R^c S'\} =1.\]
 Therefore,
\[\V_S^c(t:A)=\E_S^c(t,A)*_L \V_S^{c\Box}(A)=|t:A|^\CS_S *_L \V_S^{c\Box}(A)= 1 *_L 1=1.\]
Hence, $\V_S^c(t:A)=|t:A|^\CS_S$.\qed
\end{proof}

\begin{lemma}
The canonical model $\M^c=(\W^c,\R^c,\E^c,\V^c)$ is a model of $\F_\CS$.
\end{lemma}
\begin{proof}
 It suffices to show that $\E^c$ satisfies conditions $F\E1-F\E3$. Suppose $S$ is an arbitrary world in $\W^c$. 
\begin{itemize}
   \item Condition $F\E1$: let $$X=\{r~|~S\vd_{\CS} \bar{r}\r s:(A\r B)\},$$ $$Y=\{r~|~S\vd_{\CS} \bar{r}\r t:A\},$$ and $$Z=\{r~|~S\vd_{\CS} \bar{r}\r s\cdot t:B\}.$$ It is easy to show that $X\cap Y\subseteq Z$, and hence $\sup(X\cap Y)\leq \sup(Z)$. On the other hand, one can verify that $X$ and $Y$ have the following property: if $r$ is in $X$ (or $Y$) and $r'\leq r$, then $r'$ is in $X$ (or $Y$). Therefore, $X\subseteq Y$ or $Y\subseteq X$. If $X\subseteq Y$, then $\sup(X\cap Y)=\sup(X)\leq \sup(Z)$, which yields $$|s:(A\r B)|^\CS_S\leq |s\cdot t:B|^\CS_S.$$ Since $x*_L y\leq x$, for all $x,y\in[0,1]$, we have  
  $$|s:(A\r B)|^\CS_S *_L |t:A|^\CS_S\leq |s\cdot t:B|^\CS_S.$$ 
  Hence 
  $$\E_S^c(s,A\r B)*_L \E_S^c(t,A)\leq\E_S^c(s\cdot t,B).$$
If $Y\subseteq X$, then $\sup(X\cap Y)=\sup(Y)\leq \sup(Z)$, which yields $$|t:A|^\CS_S\leq |s\cdot t:B|^\CS_S.$$ 
Since $x*_L y\leq y$, for all $x,y\in[0,1]$, we have    
  $$|s:(A\r B)|^\CS_S *_L |t:A|^\CS_S\leq |s\cdot t:B|^\CS_S.$$  Hence 
  $$\E_S^c(s,A\r B)*_L \E_S^c(t,A)\leq\E_S^c(s\cdot t,B).$$
     \item Condition $F\E2$: it is easy to show that 
     $$\{r| S\vd_\CS \bar{r}\r s:A\}\subseteq \{r| S\vd_\CS \bar{r}\r s+t:A\},$$
  $$\{r| S\vd_\CS \bar{r}\r s:A\}\subseteq \{r| S\vd_\CS \bar{r}\r t+s:A\}.$$
   Thus $$|s:A|^\CS_S \leq |s+t:A|^\CS_S,$$ $$|s:A|^\CS_S \leq |t+s:A|^\CS_S.$$ Hence $$\E_S^c(s,A)\leq \E_S^c(s+t,A),$$ 
   $$\E_S^c(s,A)\leq \E_S^c(t+s,A).$$ 
    \item Condition $F\E3$: Suppose $c \overset{1}{:} F \in \CS$. Then, by GIAN, we have $\vdash_\CS c \overset{1}{:} F$. Thus, by Lemma \ref{lem:theorems of RPLJ} item 7, $\vdash_\CS c:F$. Hence, $\vdash_\CS \bar{1}\r c:F$. Therefore, 
    \[\E_S^c(c,F) = |c:F|_S^\CS = \sup \{r|S \vdash_\CS \bar{r} \r c:F\} =1.\]
    This completes the proof.\qed
\end{itemize}
\end{proof}

\begin{theorem}[Graded-style completeness]
For each set of formulas $T$, formula $A$ and constant specification $\CS$:
\[\| A\|^\CS_T =|A|^\CS_T.\]
\end{theorem}
\begin{proof}
We first prove that $|A|^\CS_T <\|A\|^\CS_T$. Suppose $|A|^\CS_T=r$. Suppose $r'$ is an arbitrary element of $\{r| T\vd_\CS \bar{r}\r A\}$, i.e. 
$$T\vd_\CS \bar{r'}\r A.$$ 
By Theorem \ref{thm:soundness RPLJ} we obtain $r'\leq \V_w(A)$, for every $\F_\CS$-model $\M$ of $T$, and every $w\in\M$. Thus $r'\leq \|A\|^\CS_T$. It follows that $r\leq \|A\|^\CS_T$, that is $|A|^\CS_T <\|A\|^\CS_T$.

For the converse, we show that for all $r\in \mathbb{Q}\cap[0,1]$ if $r<\|A\|^\CS_T$ then $T\vd_\CS \bar{r}\r A$ (this implies that $\|A\|^\CS_T < |A|^\CS_T$), or rather its contrapositive. Suppose $T\not\vd_\CS \bar{r}\r A$. Hence $T\cup\{A\r \bar{r}\}$ is $\CS$-consistent, and can be extended to a $\CS$-consistent $\CS$-complete set $S$. Now consider the canonical model $\M^c=(\W^c,\R^c,\E^c,\V^c)$ of $\F_\CS$. Clearly $S\in\W^c$.  By the Truth Lemma (Lemma \ref{lem:Truth Lemma}), $\M^c$ is a model of $S$, and thus $\V_S^c(A\r\bar{r})=1$. Therefore, $\V_S^c(A)\leq r$, and hence $r\not<\|A\|^\CS_T$.\qed 
\end{proof}

\begin{corollary}[Fuzzy completeness]
For each set of formulas $T$, formula $A$ and constant specification $\CS$, if $T\models_\CS A$ then for every $n\geq 1$:
$$T\vdash_\CS \overline{(1-\frac{1}{n})}\r A.$$
\end{corollary}
\begin{proof}
Note that $T\models_\CS  A$ implies that $\| A\|^\CS_T=1$. Thus, by the graded-style completeness, $| A|^\CS_T=1$. Therefore, for every $n\geq 1$ there is $r\in \mathbb{Q}\cap[0,1]$ such that $T\vd_\CS \bar{r}\r A$ and $1-\frac{1}{n} \leq r <1$. Since $\vdash_\CS \overline{(1-\frac{1}{n})}\r \bar{r}$, by axiom BL1 it follows that $T\vdash_\CS \overline{(1-\frac{1}{n})}\r A.$\qed
\end{proof}

\begin{lemma}\label{lem:every cons set has model RPLJ}
Every $\CS$-consistent set of \F-formulas have a $\F_\CS$-model.
\end{lemma}
\begin{proof}
Suppose $T$ is a $\CS$-consistent set of \F-formulas. Then it can be extended to a $\CS$-consistent and $\CS$-complete set $S$. Now consider the canonical model $\M^c=(\W^c,\R^c,\E^c,\V^c)$ of $\F_\CS$. Clearly $S\in\W^c$.  By the Truth Lemma (Lemma \ref{lem:Truth Lemma}), $\M^c$ is a model of $S$, and thus is a model of $T$.\qed
\end{proof}

\begin{lemma}[Compactness]
For each set of \F-formulas $T$, if each finite $T_0\subseteq T$ has a $\F_\CS$-model then $T$ has a $\F_\CS$-model.
\end{lemma}
\begin{proof}
Suppose that $T$ has no $\F_\CS$-model. Then, by Lemma \ref{lem:every cons set has model RPLJ}, $T$ should be $\CS$-inconsistent. Hence, a finite subset $T_0\subseteq T$ is $\CS$-inconsistent, and thus has no $\F_\CS$-model.\qed
\end{proof}

\begin{lemma}[Conservativity]
Let $A$ be a formula in the language of Rational Pavelka logic {\sf RPL}. If $A$ is provable in  $\F_\CS$, then it is provable in $\RPL$.
\end{lemma}
\begin{proof}
Suppose $A$ is a formula in the language of {\sf RPL}, and it is provable in  $\F_\CS$.  Now suppose that $\V$ is an arbitrary truth valuation of {\sf RPL}. It suffices to show that  $\V(A)=1$. Then by completeness of {\sf RPL} (Theorem \ref{thm:strong completeness RPL}), ${\sf RPL}\vdash A$. 

Now I aim to construct an $\F_\CS$-model from the valuation $\V$ of \RPL. Let $\M=(\W,\R,\E,\V')$ be defined follows:
\begin{itemize}
\item $\W=\{w_0 \}$.
\item $\R=\emptyset$.
\item $\E_{w_0}(t,A)=1$, for all $t\in Tm$ and $\F$-formula $A$.
\item $\V'_{w_0}(p)=\V(p)$, for all $p\in\mathcal{P}$.
\end{itemize}
Clearly, $\E$ satisfies the conditions $F\E1-F\E3$, and thus $\M$ is an $\F_\CS$-model. By induction on the complexity of $A$ it can be easily shown that $\V(A)=\V_{w_0}(A)$. Since $A$ is provable in  $\F_\CS$, then by Lemma \ref{thm:soundness RPLJ}, $A$ is $\F_\CS$-valid, that is $\V_w(A)=1$ for every $\F_\CS$-model $\M=(\W,\R,\E,\V)$ and every $w\in\W$. Thus, $\V_{w_0}(A)=1$, and hence $\V(A)=1$.
 \qed
\end{proof}

In fact the model constructed in the above proof is a Mkrtychev model of $\F$. Similar to M-models of {\sf BLJ} we can define M-models for $\F$. All the results of this section holds for these M-models too.
\section{Other justification principles}\label{sec:Other justification principles}
Another important principle in justification logics is the \textit{factivity axiom},  referred to as the axiom jT:
\[ t:A \r A.\]
In classical justification logics axiom jT states that justifications are factive: every justified statement is true. Let us see how the addition of  axiom jT to our fuzzy justification logics {\sf BLJ}, {\sf {\L}J}, {\sf GJ}, ${\sf \Pi J}$, and $\F$ affects matter (specially the proof of soundness theorems).

In the framework of fuzzy justification logic, at first sight it seems that axiom jT is unacceptable. For weak justifications (justifications with low certainty degree) for a statement, would not necessarily imply the truth of that statement. But note that in the framework of t-norm based fuzzy logics, validity of jT means that the truth value of $t:A$ is less than or equal to the truth value of $A$, which holds in all  fuzzy Fitting models that have a reflexive  accessibility relation $\R$.\footnote{It is worth noting that the modal axiom $\Box A\r A$ is admitted in many fuzzy modal logics, see e.g. \cite{Hajek1998,Hajek2010b,Priest2008}.} The proof is as follows. Let $\M=(\W,\R,\E,\V)$ be an arbitrary fuzzy Fitting model, $w\in\W$ be an arbitrary world, and $\R$ be reflexive. Then 
\[ \V_w^\Box(A) =\inf\{\V_v(A)|w \R v\} \leq \V_w(A).\]
Thus, 
\[ \V_w(t:A) = \E_w(t,A) *_L \V_w^\Box(A) \leq   \V_w^\Box(A)  \leq  \V_w(A). \]

By adding axiom jT to $\F$, the followings are provable:
\begin{enumerate}
\item  $t \overset{1}{:} A \r A$.
\item $t:_r A\rightarrow (\bar{r}\r A)$.
\item  $\neg t:\bar{0}$.
\end{enumerate}

Instead of axiom jT, either of the above formulas can be added to our fuzzy justification logics. The last one is an instance of axiom jT, $t:\bar{0}\r \bar{0}$, and is called axiom jD. This axiom is acceptable if we want to axiomatize consistent belief (i.e. it is impossible to hold contradictory belief), or normative belief (in which the notion of ``ought to believe" is concerned).

It is easy to see that jD is valid in those fuzzy Fitting models of {\sf {\L}J}, {\sf GJ}, ${\sf \Pi J}$, or $\F$ that have a serial accessibility relation $\R$. The proof is as follows. Let $\M=(\W,\R,\E,\V)$ be an arbitrary fuzzy Fitting model, $w\in\W$ be an arbitrary world, and $\R$ be serial. Then, 
\[ \V_w^\Box(\bar{0}) =\inf\{\V_v(\bar{0})|w \R v\} = 0.\]
Thus, $\V_w(t:\bar{0}) =0$. Note that in all logics {\sf {\L}J}, {\sf GJ}, ${\sf \Pi J}$, and $\F$ we have that $\V_w(\neg A)=1$, whenever $\V_w(A)=0$. Hence, $\V_w(\neg t:\bar{0}) =1$.

We leave the study of adding other justification principles to our fuzzy systems for future work.
\section{Conclusion}
We extended fuzzy logic and justification logic to fuzzy justification logic, a framework for reasoning on \textit{justifications under vagueness}. We replaced the classical base of the justification logic {\sf J} with some known fuzzy logics: Hajek's basic logic, {\L}ukasiewicz logic, G\"{o}del logic, product logic, and rational Pavelka logic. In all of the resulting systems we introduced fuzzy models (fuzzy possible world semantics with crisp accessibility relation) for our systems, and established the soundness theorems. In the extension of rational Pavelka logic we also proved a graded-style completeness theorem.

This paper is the first study of fuzzy justification logics, and there are many directions for further work. The main problem remain open in this paper is the proof of the standard completeness theorem (i.e. if a formula is valid, then it is provable). Another direction for further research is to study fuzzy versions of Fitting models with soft accessibility relations $\R:\W\times\W\r [0,1]$. Kripke models of this kind have been studied in the literature for fuzzy modal logics, initiating from the work of Fitting \cite{Fitting1991}. The relationships between our fuzzy versions of {\sf J} and fuzzy versions of modal logic {\sf K} is also interesting (the realization theorem).\\

\noindent
{\bf Acknowledgments}\\

 I would like to thank Mohammad Amin Khatami and  Nazanin Roshandel-Tavana for their useful arguments. Many thanks to Morteza Moniri for bringing the rational Pavelka logic to my attention, and for many discussions and helpful suggestions. This research was in part supported by a grant from IPM. (No. 93030416)

\end{document}